\documentclass[12pt]{amsart}

\usepackage{amsmath, amscd, amssymb}

\input xy
\xyoption{all}

\newcommand{\bC}{{\mathbb C}}

\newcommand{\bP}{{\mathbb P}}

\newcommand{\bp}{{\mathbf p}}
\newcommand{\by}{{\mathbf y}}

 \newcommand{\cP}{{\mathcal P}}

\newcommand{\cC}{{\mathcal C}}

\newcommand{\Mbar}{\overline{\mathcal M}}

\newcommand{\pd}{\partial}

\DeclareMathOperator{\Aut}{Aut}

\newtheorem{theorem}{Theorem}[section]

\newtheorem{proposition}[theorem]{Proposition}

\theoremstyle{remark}

\theoremstyle{definition}

\newcommand{\bea}{\begin{eqnarray}}
\newcommand{\eea}{\end{eqnarray}}
\newcommand{\ben}{\begin{eqnarray*}}
\newcommand{\een}{\end{eqnarray*}}
\newcommand{\be}{\begin{equation}}
\newcommand{\ee}{\end{equation}}

\begin{document}

\title
{Quantum Mirror Curves for $\bC^3$ and the Resolved Confiold}

\author{Jian Zhou}
\address{Department of Mathematical Sciences\\Tsinghua University\\Beijing, 100084, China}
\email{jzhou@math.tsinghua.edu.cn}

\begin{abstract}
We establish a conjecture of Gukov and Su{\l}kowski
in the following three cases:
Lambert curve for Hurwitz numbers,
framed mirror curve of $\bC^3$,
and the framed mirror curve of the resolved conifold.
\end{abstract}

\maketitle



\vspace{.2in}

In this   sequel to \cite{ZhoAiry}
we treat three more cases of a conjecture made by
Gukov and Su{\l}kowski \cite{Guk-Sul}.
We will study the following open string invariants:
Hurwitz numbers,
one-legged topological vertex,
and the resolved conifold with one framed outer D-brane.
By the BKMP remodeling conjecture \cite{Mar, BKMP, Bou-Mar},
in each of these cases,
the open string invariants can be encoded in some curve
on which one can define some differentials
by the Eynard-Orantin recursion \cite{Eyn-OraInv}.
For the proofs in these cases,
see \cite{BEMS, EMS, Che, ZhoMV, Eyn-Ora}.
In each case the mirror curve are given by an equation
$$A(u,v) = 0, \;\;\; u,v \in \bC $$
or
$$A(x,y) = 0, \;\;\; x, y \in \bC^*.$$
Gukov and Su{\l}kowski \cite{Guk-Sul} defined some partition function $Z$
and conjectured that there is a quantization
$\hat{A}(\hat{x},\hat{y})$ of $A(x,y)$ into differential operator
such that
\be
\hat{A}(\hat{x}, \hat{y}) Z = 0.
\ee
In \cite{ZhoAiry} we have established this in the case of the Airy curve.
In this paper we will deal with the three cases mentioned above.
We will use the known results of the corresponding A-model calculations
in each case,
and derivation of the quantum mirror curves
follows that of \cite[(6.7)-(6.12)]{Guk-Sul}.

\section{Hurwitz Numbers and Quantum Lambert Curve}

\label{sec:Lambert}

\subsection{Hurwitz numbers and Burnside formula}

For a partition $\mu = (\mu_1, \dots, \mu_{l(\mu)})$ of $d>0$,
denote by $H_{g, \mu}$ the Hurwitz number of branched coverings of
$\bP^1$ of type $\mu$ by genus $g$ Riemann surfaces.
In general,
these numbers can be computed by the Burnside formula:
\be \label{eqn:Burnside}
\exp\sum_{\mu \neq \emptyset}
\sum_{g \geq 0} \frac{\lambda^{2g-2+l(\mu)+|\mu|}}{(2g-2+l(\mu)+|\mu|)!} H_{g, \mu} p_\mu
= \sum_\nu \frac{\dim R_\nu}{|\nu|!} \cdot e^{\kappa_\nu \lambda/2} \cdot s_\nu.
\ee
Here  $p_\mu = \prod_{i=1}^{l(\mu)} p_{\mu_i}$ are the Newton functions,
$s_\nu$ are the Schur functions.
They are related to each other
by the characters of irreducible representations of the symmetric groups:
\be \label{eqn:SchurNewton}
s_\mu = \sum_\nu  \frac{\chi_\mu(\nu)}{z_\nu} p_\nu, \;\;\;\;
p_\nu = \sum_\mu \chi_\mu(\nu) s_\mu,
\ee
where $\chi_\mu$ denotes the character of the irreducible representation $R_\mu$ indexed by $\mu$
and $\chi_\mu(\nu)$ denotes its value on the conjugacy class indexed by $\nu$.
For a partition $\mu = (\mu_1, \dots, \mu_l)$,
the number $\kappa_\mu$ is defined as follows:
\be
\kappa_\mu = \sum_{i=1}^l \mu_i(\mu_i-2i+1).
\ee

\subsection{Hurwitz numbers and the cut-and-join equation}

Denote by $H^\bullet$ the left-hand side of \eqref{eqn:Burnside}.
It satisfies the following differential equation, called
the {\em  cut-and-join equation} \cite{Gou-Jac-Vai}:
\be \label{eqn:CutJoin}
\frac{\pd H^\bullet}{\partial\lambda} = KH^\bullet,
\ee
where
\be
K =\frac{1}{2}\sum_{i,j\geq 1}\left(ijp_{i+j}\frac{\pd^2}{\pd p_i\pd p_j}
+(i+j)p_ip_j\frac{\pd}{\pd p_{i+j}}\right).
\ee
This is because
\be
K s_\nu = \frac{1}{2}\kappa_\nu \cdot s_\nu.
\ee

\subsection{Hurwitz numbers and ELSV formula}

Hurwitz numbers are related to Hodge integrals on the Deligne-Mumford
moduli spaces by the ELSV formula \cite{ELSV}:
\be \label{eqn:ELSV}
H_{g, \mu} = \frac{1}{|\Aut(\mu)|} \prod_{i=1}^{l(\mu)} \frac{\mu_i^{\mu_i}}{\mu_i!}
\int_{\Mbar_{g,l(\mu)}} \frac{\Lambda_g^{\vee}(1)}{\prod_{i=1}^{l(\mu)} (1 - \mu_i \psi_i)},
\ee
where $\Lambda_g^{\vee}(1) = \sum_{i=0}^g (-1)^i \lambda_i$.

\subsection{Symmetrization}
One can also define
\ben
&& H_g(\bp)
= \sum_{\mu} \frac{1}{(2g-2+l(\mu)+|\mu|)!} H_{g, \mu}p_{\mu}.
\een
Because $H_g(p)$ is a formal power series in $p_1, p_2, \dots, p_n, \dots$,
for each $n$,
one can obtain from it a formal power series $\Phi_{g,n}(x_1, \dots, x_n)$
by applying the following linear symmetrization operator \cite{Gou-Jac-Vai}:
$$p_{\mu} \mapsto \delta_{l(\mu), n}\sum_{\sigma \in S_n}
x_{\sigma(1)}^{\mu_1} \cdots x_{\sigma(n)}^{\mu_n}.$$
Because
\ben
&& \sum_{\nu} e^{\kappa_\nu\lambda/2} \cdot \frac{\dim R_\nu}{|\nu|!} \cdot s_\nu
= 1 + \sum_{|\nu|\geq 1} e^{\kappa_\nu \lambda/2} \frac{\dim R_\nu}{|\nu|!} \sum_{\mu}
\frac{\chi_\nu(\mu)}{z_\mu} p_\mu,
\een
we have
\ben
&& \log \sum_{\nu} e^{\kappa_\nu\lambda/2} \frac{\dim R_\nu}{|\nu|!} \cdot s_\nu \\
& = & \sum_{k \geq 1} \frac{(-1)^{k-1}}{k}
\sum_{\substack{\mu^1 \cup \cdots \cup \mu^k = \mu\\ |\mu^1|, \dots, |\mu^k| > 0}}
\prod_{i=1}^k \sum_{|\nu^i| =|\mu^i|} \sum_{|\nu^i|=|\mu^i|}
e^{\kappa_{\nu^i}\lambda/2} \frac{\dim R_{\nu^i}}{|\nu^i|!}
\frac{\chi_{\nu^i}(\mu^i)}{z_{\mu^i}} \cdot p_\mu.
\een
It follows that
\be \label{eqn:Hgn}
\begin{split}
& \sum_{g \geq 0} \frac{\lambda^{2g-2+l(\mu)+|\mu|}}{(2g-2+l(\mu)+|\mu|)!} H_{g, \mu} p_\mu \\
= & \sum_{k \geq 1} \frac{(-1)^{k-1}}{k}
\sum_{\substack{\mu^1 \cup \cdots \cup \mu^k = \mu\\ |\mu^1|, \dots, |\mu^k| > 0}}
\prod_{i=1}^k \sum_{|\nu^i|=|\mu^i|}e^{\kappa_{\nu^i}/2} \frac{\dim R_{\nu^i}}{|\nu^i|!}
\frac{\chi_{\nu^i}(\mu^i)}{z_{\mu^i}} \cdot p_\mu.
\end{split}
\ee
After the symmetrization,
\be
\begin{split}
& \sum_{g \geq 0} \lambda^{2g-2+n} \Phi_{g,n}(x_1, \dots ,x_n) \\
= & \sum_{k \geq 1} \frac{(-1)^{k-1}}{k}
\sum_{\substack{l(\mu^1) + \cdots + l(\mu^k) = n \\ |\mu^1|, \dots, |\mu^k| > 0}}
\prod_{i=1}^k \sum_{|\nu^i|=|\mu^i|}e^{\kappa_{\nu^i}\lambda/2} \frac{\dim R_{\nu^i}}{|\nu^i|!} \cdot
\frac{\chi_{\nu^i}(\mu^i)}{z_{\mu^i}} \\
& \cdot \lambda^{-|\mu|} \sum_{\sigma \in S_n} x_{\sigma(1)}^{\mu_1} \cdots x_{\sigma(n)}^{\mu_n},
\end{split}
\ee
where $\mu = (\mu_1, \dots, \mu_n) = \mu^1 \cup \cdots \cup \mu^k$.
In particular,
\be \label{eqn:F}
\begin{split}
& \frac{1}{n!} \sum_{g \geq 0} \lambda^{2g-2+n}  \Phi_{g,n}(x, \dots ,x) \\
= & \sum_{k \geq 1} \frac{(-1)^{k-1}}{k}
\sum_{\substack{\sum_{i=1}^k l(\mu^i) = n \\ |\mu^1|, \dots, |\mu^k| > 0}}
\prod_{i=1}^k \sum_{|\nu^i|=|\mu^i|}e^{\kappa_{\nu^i}\lambda/2}
\frac{\dim R_{\nu^i}}{|\nu^i|!}
\frac{\chi_{\nu^i}(\mu^i)}{z_{\mu^i}} (\frac{x}{\lambda})^{|\mu^i|}.
\end{split}
\ee

\subsection{Computation of a partition function}

\begin{proposition}
Let
\be
Z =
\exp \sum_{n \geq 1} \frac{1}{n!} \sum_{g \geq 0} \lambda^{2g-2+n}
\Phi_{g,n}(x, \dots ,x).
\ee
Then one has
\be
Z = \sum_{n=0}^\infty e^{n(n-1)\lambda/2} \frac{x^n}{n!\lambda^n}.
\ee
\end{proposition}

\begin{proof}
By \eqref{eqn:F} we have
\ben
Z
& = & 1+ \sum_{|\mu| > 0}
\sum_{|\nu|=|\mu|} e^{\kappa_{\nu}\lambda/2} \frac{\dim R_{\nu}}{|\nu|!}
\frac{\chi_{\nu}(\mu)}{z_{\mu}} \cdot (x/\lambda)^{|\mu|}  \\
& = & 1+ \sum_{n=1}^\infty \sum_{|\nu| = n} (x/\lambda)^{|\nu|}
e^{\kappa_{\nu}\lambda/2} \frac{\dim R_{\nu}}{|\nu|!}
\sum_{|\mu|=n} \frac{\chi_{\nu}(\mu)}{z_{\mu}}  \\
& = & 1+ \sum_{n=1}^\infty \sum_{|\nu| = n} (x/\lambda)^{|\nu|}
e^{\kappa_{\nu}\lambda/2} \frac{\dim R_{\nu}}{|\nu|!} \delta_{\nu, (n)}  \\
& = & 1+ \sum_{n=1}^\infty  (x/\lambda)^n
e^{\kappa_{(n)}\lambda/2} \frac{\dim R_{(n)}}{n!} \\
& = & \sum_{n=0}^\infty e^{n(n-1)\lambda/2} \frac{x^n}{n!\lambda^n}.
\een
\end{proof}

\subsection{Differential equation satisfied by the partition function}
Write $Z= a_0 + a_1 + \cdots$, where
\be
a_n = e^{n(n-1)\lambda/2} \frac{x^n}{n!\lambda^n}.
\ee
Then one has
\be
\frac{a_{n+1}}{a_n} = \frac{x}{(n+1)\lambda} e^{n\lambda},
\ee
and so
\be
(n+1)\lambda a_{n+1} - x e^{n\lambda} a_n = 0.
\ee
By summing over $n$,
one gets:

\begin{theorem}
For the partition function associated with the Hurwitz numbers
define above,
the following equation is satisfied:
\be
(\hat{y} - \hat{x} e^{\hat{y}} ) Z = 0,
\ee
where
\be
\hat{x} = x \cdot, \;\;\; \hat{y} = \lambda x \frac{\pd}{\pd x}.
\ee
\end{theorem}

This established the  Lambert curve case of
Gukov-Su{\l}kowski conjecture \cite{Guk-Sul}.
Recall that Bouchard and Mari\~no \cite{Bou-Mar} conjectured that
\be
W_{g,n}(x_1, \dots, x_n) = \pd_{x_1} \cdots \pd_{x_n}
\Phi_{g,n}(x_1, \dots, x_n) dx_1 \cdots dx_n
\ee
satisfies the Eynard-Orantin recursion
defined by the Lambert curve:
\be
A(x,y) = y - xe^y = 0.
\ee
This is sort of a limiting case of the $\bC^3$ case of the BKMP conjecture \cite{BKMP}.
This conjecture has been proved by Borot-Eynard-Safnuk-Mulase \cite{BEMS}
and Eynard-Safnuk-Mulase \cite{EMS} by two different methods.
Our result shows that the quantization
of $A(x,y)$ does not involve higher order quantum corrections
in this case.

\section{Mari\~no-Vafa Formula and Quantum Mirror Curve of $\bC^3$}

\subsection{Mari\~no-Vafa formula}

For a partition $\mu = (\mu_1, \dots, \mu_{l(\mu)})$,
consider the triple Hodge integral:
\begin{eqnarray*}
\cC_{g, \mu}(a)& = & - \frac{\sqrt{-1}^{|\mu|+l(\mu)}}{|\Aut(\mu)|}
 (a(a+1))^{l(\mu)-1}
\prod_{i=1}^{l(\mu)}\frac{ \prod_{a=1}^{\mu_i-1} (\mu_i a+a)}{(\mu_i-1)!} \\
&& \cdot \int_{\Mbar_{g, l(\mu)}}
\frac{\Lambda^{\vee}_g(1)\Lambda^{\vee}_g(-a-1)\Lambda_g^{\vee}(a)}
{\prod_{i=1}^{l(\mu)}(1- \mu_i \psi_i)}.
\een
Note when $l(\mu) \geq 3$,
we have
\be
\int_{0, l(\mu)} \frac{\Lambda^\vee_0(1) \Lambda_0^\vee(a)\Lambda_0^\vee(-1-a)}
{\prod\limits_{i=1}^{l(\mu)} (1-\mu_i \psi_i)}
= \int_{0, l(\mu)} \frac{1}{\prod\limits_{i=1}^{l(\mu)} (1-\mu_i \psi_i)}
= |\mu|^{l(\mu)-3},
\ee
We use this to extend the definition to the case of $(g,n) = (0,1)$ and $(0,2)$.
The Mari\~no-Vafa formula \cite{Mar-Vaf, LLZ, Oko-Pan}  states that
\be
\sum_{|\mu|\geq 1} \sum_{g \geq 0} \lambda^{2g-2+l(\mu)}\cC_{g, \mu}(a) p_\mu
= \log \sum_{|\nu|\geq 0} q^{a\kappa_\nu\tau/2} \sqrt{-1}^{|\nu|} s_{\nu}(q^\rho) s_\nu,
\ee
where $q=e^{\sqrt{-1}\lambda}$,
$s_\nu(q^\rho) = s_\nu(q^{-1/2}, q^{-3/2}, \dots)$.
Equivalently,
\be \label{eqn:Cgn}
\begin{split}
& \sum_{g \geq 0} \lambda^{2g-2+l(\mu)}\cC_{g, \mu}(a)\\
= & \sum_{k \geq 1} \frac{(-1)^{k-1}}{k}
\sum_{\substack{\mu^1 \cup \cdots \cup \mu^k = \mu\\ |\mu^1|, \dots, |\mu^k| > 0}}
\prod_{i=1}^k \sum_{|\nu^i|=|\mu^i|}q^{a\kappa_{\nu^i}/2} \sqrt{-1}^{|\nu|} s_{\nu^i}(q^\rho)
\frac{\chi_{\nu^i}(\mu^i)}{z_{\mu^i}}.
\end{split}
\ee

\subsection{Symmetrization}
For fixed $g \geq 0$ and $n \geq 1$,
define
\ben
&& \cC_{g,n}(p;a)
= \sum_{l(\mu) = n} \cC_{g, \mu}(a) p_{\mu}.
\een
Because $\cC_{g,n}(p;a)$ is a formal power series in $p_1, p_2,  \dots$,
for each $n$,
one can obtain from it a formal power series $\Phi_{g,n}(x_1, \dots, x_n; a)$
by applying the following linear symmetrization operator \cite{Gou-Jac-Vai, Che}:
$$p_{\mu} \mapsto (\sqrt{-1})^{-(n+|\mu|)}
\delta_{l(\mu), n}\sum_{\sigma \in S_n}
x_{\sigma(1)}^{\mu_1} \cdots x_{\sigma(n)}^{\mu_n}.$$
Hence by \eqref{eqn:Cgn} we have
\ben
&& \sum_{g \geq 0} \lambda^{2g-2+n} \Phi_{g,n}(x_1, \dots ,x_n; a) \\
& = & \sum_{k \geq 1} \frac{(-1)^{k-1}}{k}
\sum_{\substack{l(\mu^1) + \cdots + l(\mu^k) = n \\ |\mu^1|, \dots, |\mu^k| > 0}}
\prod_{i=1}^k \sum_{|\nu^i|=|\mu^i|}q^{a\kappa_{\nu^i}/2} \sqrt{-1}^{|\nu|} s_{\nu^i}(q^\rho)
\frac{\chi_{\nu^i}(\mu^i)}{z_{\mu^i}} \\
&& \cdot (\sqrt{-1})^{-(n+|\mu^1|+\cdots +|\mu^k|)}
\sum_{\sigma \in S_n} x_{\sigma(1)}^{\mu_1} \cdots x_{\sigma(n)}^{\mu_n},
\een
where $\mu = (\mu_1, \dots, \mu_n) = \mu^1 \cup \cdots \cup \mu^k$.
In particular,
\be \label{eqn:FC3}
\begin{split}
& \frac{1}{n!} \sum_{g \geq 0} \lambda^{2g-2+n} \sqrt{-1}^n \Phi_{g,n}(x, \dots ,x; a) \\
= & \sum_{k \geq 1} \frac{(-1)^{k-1}}{k}
\sum_{\substack{l(\mu^1) + \cdots + l(\mu^k) = n \\ |\mu^1|, \dots, |\mu^k| > 0}}
\prod_{i=1}^k \sum_{|\nu^i|=|\mu^i|}q^{a\kappa_{\nu^i}/2} s_{\nu^i}(q^\rho)
\frac{\chi_{\nu^i}(\mu^i)}{z_{\mu^i}} \\
& \cdot  x^{|\mu^1|+ \cdots +|\mu^k|}.
\end{split}
\ee

\subsection{Computation of the partition function}

\begin{proposition}
For the framed local $\bC^3$,
define
\be
Z = \exp \sum_{n \geq 1} \frac{1}{n!} \sum_{g \geq 0} (-1)^{g-1+n}
\lambda^{2g-2+n}\Phi_{g,n}(x, \dots ,x; a).
\ee
Then one has
\be
Z =  \sum_{n=0}^\infty \frac{e^{-an(n-1)\lambda/2+n\lambda/2}}{\prod_{j=1}^n (1 - e^{j\lambda})}x^n.
\ee
\end{proposition}

\begin{proof}
By \eqref{eqn:FC3} we have
\ben
&& \exp \sum_{n \geq 1} \frac{1}{n!} \sum_{g \geq 0} \sqrt{-1}^n
\lambda^{2g-2+n}\Phi_{g,n}(x, \dots ,x; a) \\
& = & \exp \biggl( \sum_{k \geq 1} \frac{(-1)^{k-1}}{k}
\sum_{|\mu^1|, \dots, |\mu^k| > 0}
\prod_{i=1}^k \sum_{|\nu^i|=|\mu^i|}q^{a\kappa_{\nu^i}/2} s_{\nu^i}(q^\rho)
\frac{\chi_{\nu^i}(\mu^i)}{z_{\mu^i}}x^{|\mu^i|} \biggr) \\
& = & 1 + \sum_{|\mu| > 0}
\sum_{|\nu|=|\mu|}q^{a\kappa_{\nu}/2} s_{\nu}(q^\rho)
\frac{\chi_{\nu}(\mu)}{z_{\mu}} \cdot x^{|\mu|}   \\
& = &  1+ \sum_{n=1}^\infty \sum_{|\nu| = n} x^{|\nu|}
q^{a\kappa_{\nu}/2} s_{\nu}(q^\rho)
\sum_{|\mu|=n} \frac{\chi_{\nu}(\mu)}{z_{\mu}}  \\
& = & 1 + \sum_{n=1}^\infty \sum_{|\nu| = n} x^{|\nu|}
q^{a\kappa_{\nu}/2} s_{\nu}(q^\rho) \delta_{\nu, (n)}  \\
& = &  1 + \sum_{n=1}^\infty x^n
q^{a\kappa_{(n)}/2} s_{(n)}(q^\rho)
=  \sum_{n=0}^\infty \frac{q^{(2a+1)n(n-1)/4}}{[n]!} x^n.
\een
In the last equality we have used the following identity:
\be
s_{(n)}(q^{\rho})
= \frac{q^{n(n-1)/4}}{[n]!}.
\ee
The proof is completed by changing $\lambda$ to $\sqrt{-1}\lambda$:
\ben
Z & = & \sum_{n=0}^\infty
\frac{e^{-(2a+1)n(n-1)\lambda/4}}{\prod_{j=1}^n (e^{-j\lambda/2} - e^{j\lambda/2})}x^n
= \sum_{n=0}^\infty \frac{e^{-an(n-1)\lambda/2+n\lambda/2}}{\prod_{j=1}^n (1 - e^{j\lambda})}x^n.
\een
\end{proof}

\subsection{Differential equation satisfied by the partition function}
Write $Z= a_0 + a_1 + \cdots$, where
\be
a_n = e^{-an(n-1)\lambda/2+n\lambda/2} \frac{x^n}{\prod_{j=1}^n (1-e^{j\lambda})}.
\ee
Then one has
\be
\frac{a_{n+1}}{a_n} = \frac{e^{\lambda/2}x}{1- e^{(n+1)\lambda}} e^{-an\lambda},
\ee
and so
\be
(1-e^{(n+1)\lambda}) a_{n+1} -e^{\lambda/2} x e^{-an\lambda} a_n = 0.
\ee
By summing over $n$,
one gets:

\begin{theorem}
For the partition function associated with the Hurwitz numbers
define above,
the following equation is satisfied:
\be
(1- \hat{y} - e^{\lambda/2}\hat{x}\hat{y}^{-a} ) Z = 0,
\ee
where
\be
\hat{x} = x \cdot, \;\;\; \hat{y} = e^{\lambda x \frac{\pd}{\pd x}}.
\ee
\end{theorem}

This proves the Gukov-Su{\l}kowski conjecture \cite{Guk-Sul}
for the framed mirror curve of $\bC^3$.
Recall the BKMP conjecture for $\bC^3$ (see \cite{Bou-Mar})
has been proved \cite{Che, ZhoMV}.
More precisely,
\ben
&& W_{g,n}(x_1, \dots, x_n) = (-1)^{g-1+n} \pd_{x_1} \cdots \pd_{x_n}
\Phi_{g,n}(x_1, \dots, x_n) dx_1 \cdots dx_n
\een
satisfies the Eynard-Orantin recursion
defined by the following algebraic curve (cf. \cite[(8)]{ZhoMV}):
\be
x - y^a + y^{a+1} = 0.
\ee
Our result shows that one needs to change the above equation to
\be
A(x,y) = 1 -y - xy^{-a} = 0
\ee
before taking the quantization.
Further,
higher order quantum corrections introduce an extra factor of $e^{\lambda/2}$ for $\hat{x}$,
i.e., one should take $\hat{x} = e^{\lambda/2} \cdot x$.

\section{Open String Invariants of the Resolved Conifold and Quantization
of its Mirror Curve}

\subsection{Open string amplitudes of the resolved conifold
with one special outer brane}

Let us first recall the result in \cite{ZhoFMV} about
the open string amplitude for the resolved conifold with one special outer brane and framing $a$
by the theory of the topological vertex \cite{AKMV, LLLZ},
corresponding to the following toric diagrams as follows:
$$
\xy
(-4,5); (0,10), **@{.}; (0,0), **@{-}; (-7, -7), **@{-};
(0,0); (20,0), **@{-}; (27,7), **@{-}; (20,0); (20, -10), **@{-};
(-2, 3)*+{\mu}; (5,1.5)*+{\nu^t}; (15,1)*+{\nu}; (11,-14)*+{\bf Figure\;\; 1.};
\endxy
$$
This is the Figure 1(b) case in \cite{ZhoFMV}.
The case of Figure 1(a) can be treated in a similar fashion.
We have
\ben
\tilde{Z}^{(a)}(\lambda; t; \bp)
= \sum_{\mu, \nu, \eta} q^{a \kappa_\mu/2}  C_{\mu, (0), \nu^t}(q) \cdot Q^{|\nu|} \cdot
C_{\nu,(0), (0)}(q) \cdot
\frac{\chi_{\mu}(\eta)}{z_{\eta}\sqrt{-1}^{l(\eta)}} p_{\eta}.
\een
where $q = e^{\sqrt{-1}\lambda}$ and $Q = - e^{-t}$.
The normalized open string amplitudes are defined by:
\begin{align*}
\hat{\tilde{Z}}^{(a)}(\lambda;t;\bp) = \frac{\tilde{Z}^{(a)}(\lambda;t; \bp)}{\tilde{Z}^{(a)}(\lambda;t; \bp)|_{\bp = {\bf 0}}}.
\end{align*}
Write
\be
\hat{\tilde{Z}}^{(a)}(\lambda;t;\bp)
= \exp \sum_{\mu \in \cP^+} \sum_{g=0}^\infty \sum_{k=0}^\infty \tilde{F}^{(a)}_{g; k; \mu} \lambda^{2g-2+l(\mu)} e^{-kt} p_{\mu},
\ee
where $\cP^+$ is the set of nonempty partitions.
The following formula is proved in \cite{ZhoFMV}
as a mathematical formulation of the full Mari\~no-Vafa Conjecture \cite{Mar-Vaf}:
\begin{multline} \label{eqn:TildeF}
\exp \sum_{\mu \in \cP^+} \sum_{g=0}^\infty  \sum_{k=0}^\infty \sqrt{-1}^{l(\mu)} \tilde{F}^{(a)}_{g; k; \mu}
\lambda^{2g-2+l(\mu)}
e^{(|\mu|/2-k)t} p_{\mu} \\
= \sum_\mu s_{\mu}\cdot  q^{a\kappa_\mu/2} \cdot \dim_q R_{\mu}.
\end{multline}
(Unfortunately $l(\mu)$ is missing from $\lambda^{2g-2+l(\mu)}$ in \cite{ZhoFMV}.)
In this formula,
$\dim_q R_\mu$ is the quantum dimension
that gives the  colored large N HOMFLY polynomials
of the unknot are given by the quantum dimension \cite[(5.4)]{Mar-Vaf}:
\be
\dim_q R_\mu = \prod_{1 \leq i < j \leq l(\mu)} \frac{[\mu_i-\mu_j+j-i]}{[j-i]}
\cdot \prod_{i=1}^{l(\mu)}
\frac{\prod_{j=1}^{\mu_i} [j-i]_{e^t}}{\prod_{j=1}^{\mu_i} [j-i+l(\mu)]}.
\ee
In \cite{ZhoFMV},
it is realized as a specialization of the Schur function as follows:
If for $n \geq 1$ one has
\be \label{eqn:Pn}
p_n(\by) = \frac{e^{nt/2} - e^{-nt/2}}{[n]},
\ee
then one has:
\be
s_\mu(\by) = \dim_q R_\mu =\prod_{x \in \mu} \frac{[c(x)]_{e^t } }{[h(x)]},
\ee
where
\be
[n]_{e^t} = e^{t/2} q^{n/2} - e^{-t/2}q^{-n/2}.
\ee
In particular,
\be
s_{[n]}(\by) = \prod_{j=1}^n \frac{e^{t/2}q^{(j-1)/2} -e^{-t/2}q^{-(j-1)/2}}{q^{j/2}-q^{-j/2}}.
\ee

\subsection{Symmetrization}

Taking logarithm on both sides of \eqref{eqn:TildeF},
we get:
\be
\begin{split}
& \sum_{g=0}^\infty   \sum_{\mu \in \cP^+}
 \sum_{k=0}^\infty \sqrt{-1}^{l(\mu)} \tilde{F}^{(a)}_{g; k; \mu} e^{-kt}
\lambda^{2g-2+l(\mu)} p_{\mu} \\
= & \sum_{k \geq 1} \frac{(-1)^{k-1}}{k}
\sum_{\substack{\mu^1 \cup \cdots \cup \mu^k = \mu\\ |\mu^1|, \dots, |\mu^k| > 0}}
\prod_{i=1}^k \sum_{|\nu^i|=|\mu^i|} e^{-|\nu^i|t/2}
q^{a\kappa_{\nu^i}/2} \dim_q R_{\nu^i}\\
& \cdot \frac{\chi_{\nu^i}(\mu^i)}{z_{\mu^i}} p_{\mu^i}.
\end{split}
\ee
We apply the following linear symmetrization operator:
$$p_{\mu} \mapsto (\sqrt{-1})^{-l(\mu)}
\sum_{\sigma \in S_{l(\mu)}}
x_{\sigma(1)}^{\mu_1} \cdots x_{\sigma(l(\mu))}^{\mu_{l(\mu)}}$$
to get:
\ben
&& \sum_{g \geq 0} \lambda^{2g-2+n} \Phi_{g,n}(x_1, \dots ,x_n) \\
& = & \sum_{k \geq 1} \frac{(-1)^{k-1}}{k}
\sum_{\substack{l(\mu^1) + \cdots + l(\mu^k) = n \\ |\mu^1|, \dots, |\mu^k| > 0}}
\prod_{i=1}^k \sum_{|\nu^i|=|\mu^i|}
 e^{-|\nu^i|t/2} q^{a\kappa_{\nu^i}/2} \dim_q R_{\nu^i} \\
&& \cdot \frac{\chi_{\nu^i}(\mu^i)}{z_{\mu^i}} \cdot
\sum_{\sigma \in S_n} x_{\sigma(1)}^{\mu_1} \cdots x_{\sigma(n)}^{\mu_n},
\een
where $\mu = (\mu_1, \dots, \mu_n) = \mu^1 \cup \cdots \cup \mu^k$.
In particular,
\be \label{eqn:Fcon}
\begin{split}
& \frac{1}{n!} \sum_{g \geq 0} \lambda^{2g-2+n} \Phi_{g,n}(x, \dots ,x) \\
= & \sum_{k \geq 1} \frac{(-1)^{k-1}}{k}
\sum_{\substack{l(\mu^1) + \cdots + l(\mu^k) = n \\ |\mu^1|, \dots, |\mu^k| > 0}}
\prod_{i=1}^k \sum_{|\nu^i|=|\mu^i|}
 e^{-|\nu^i|t/2} q^{a\kappa_{\nu^i}/2} \dim_q R_{\nu^i} \\
& \cdot \frac{\chi_{\nu^i}(\mu^i)}{z_{\mu^i}} x^{|\mu^i|}.
\end{split}
\ee

\subsection{Computation of the partition function}

\begin{proposition}
For the framed local $\bC^3$,
define
\be
Z = \exp \sum_{n \geq 1} \frac{1}{n!} \sum_{g \geq 0}
\lambda^{2g-2+n}\Phi_{g,n}(x, \dots ,x).
\ee
Then one has
\be
Z = \sum_{n=0}^\infty  \prod_{j=1}^n \frac{1-e^{-t}q^{-(j-1)}}{1-q^{-j}}
q^{an(n-1)/2-n/2} x^n.
n\ee
\end{proposition}

\begin{proof}
By \eqref{eqn:Fcon} we have
\ben
&& \exp \sum_{n \geq 1} \frac{1}{n!} \sum_{g \geq 0}
\lambda^{2g-2+n}\Phi_{g,n}(x, \dots ,x) \\
& = & 1 + \sum_{|\mu| > 0}
\sum_{|\nu|=|\mu|} e^{-|\nu|t/2} q^{a\kappa_{\nu}/2} \dim_q R_{\nu}
\frac{\chi_{\nu}(\mu)}{z_{\mu}} \cdot x^{|\mu|}   \\
& = &  1+ \sum_{n=1}^\infty \sum_{|\nu| = n} x^{|\nu|}
 e^{-|\nu|t/2} q^{a\kappa_{\nu}/2} \dim_q R_{\nu}
\sum_{|\mu|=n} \frac{\chi_{\nu}(\mu)}{z_{\mu}}  \\
& = & 1 + \sum_{n=1}^\infty \sum_{|\nu| = n} x^{|\nu|}
 e^{-|\nu|t/2} q^{a\kappa_{\nu}/2} \dim_q R_{\nu} \delta_{\nu, (n)}  \\
& = &  1 + \sum_{n=1}^\infty x^n  e^{-nt/2}
q^{a\kappa_{(n)}/2} \dim_q R_{(n)} \\
& = & \sum_{n=0}^\infty  e^{-nt/2}\prod_{j=1}^n \frac{e^{t/2}q^{(j-1)/2} -e^{-t/2}q^{-(j-1)/2}}{q^{j/2}-q^{-j/2}} q^{an(n-1)/2} x^n \\
& = & \sum_{n=0}^\infty  \prod_{j=1}^n \frac{e^{-t}-q^{j-1}}{1-q^j}
\cdot q^{an(n-1)/2+n/2} x^n.
\een
\end{proof}

\subsection{Differential equation satisfied by the partition function}
Write $Z= a_0 + a_1 + \cdots$, where
\be
a_n =  \sum_{n=0}^\infty  \prod_{j=1}^n \frac{e^{-t}-q^{j-1}}{1-q^j}
\cdot q^{an(n-1)/2+n/2} x^n.
\ee
Then one has
\be
\frac{a_{n+1}}{a_n} = \frac{e^{-t}-q^n}{1- q^{n+1}} q^{an+1/2} x,
\ee
and so
\be
(1-q^{n+1}) a_{n+1} + x q^{(a+1)n+1/2} a_n - e^{-t} q^{an+1/2} x a_n = 0.
\ee
By summing over $n$,
one gets:

\begin{theorem}
For the partition function associated with the Hurwitz numbers
define above,
the following equation is satisfied:
\be
(1- \hat{y} + q^{1/2}\hat{x}\hat{y}^{a+1}
- q^{1/2}e^{-t} \hat{x} \hat{y}^{a} ) Z = 0,
\ee
where
\be
\hat{x} = x \cdot, \;\;\; \hat{y} = e^{-\sqrt{-1}\lambda x \frac{\pd}{\pd x}}.
\ee
\end{theorem}
This proves the Gukov-Su{\l}kowski conjecture \cite{Guk-Sul}
for the framed mirror curve for the resolved conifold.
In \cite{ZhoCon} it has been proved that
by counting the disc invariants one can get the following equation of
the framed mirror curve of the resolved conifold
 with an outer brane and framing $a$:
\be \label{eqn:MirrorAFraming}
y+xy^{-a} - 1 - e^{-t}xy^{-a-1} = 0.
\ee
By changing $x$ to $-x$ and $a$ to $-a-1$,
one gets the following equation:
\be
A(x,y) = 1- y + x^{a+1}- e^{-t}xy^a = 0.
\ee
Our result indicates that when taking the quantization
higher order quantum corrections introduce an extra factor of $q^{1/2}$
for $\hat{x}$,
i.e., one should take $\hat{x} = q^{1/2} \cdot x$.

\vspace{.1in}

{\em Acknowledgements}.
The author is partially supported by  NSFC grant 1171174.


\begin{thebibliography}{999}



\bibitem{AKMV}
M. Aganagic, A. Klemm, M. Mari\~no, C. Vafa,
{\em The topological vertex},
Commun. Math. Phys. {\bf 254} (2005), 425-478, arXiv:hep-th/0305132.


\bibitem{BEMS} G.~Borot, B.~Eynard,
M.~Mulase,B.~Safnuk, {Hurwitz numbers, matrix models and topological recursion},
arXiv:0906.1206.

\bibitem{BKMP}
V. Bouchard, A. Klemm, Marcos Mari\~no,  S. Pasquetti,
{\em Remodeling the B-model},
arXiv:0709.1453.

\bibitem{Bou-Mar}
V. Bouchard, M. Mari\~no,
{\em Hurwitz numbers, matrix models and enumerative geometry},
arXiv:0709.1458.


\bibitem{Che}
L. Chen,
{\em  Symmetrized cut-join equation of Marino-Vafa formula},
arXiv:0709.1738.

\bibitem{Che2}
L. Chen,
{\em Bouchard-Klemm-Mari\~no-Pasquetti Conjecture for $\bC^3$},
arXiv:0910.3739.

\bibitem{ELSV} T. Ekedahl, S. Lando, M. Shapiro, A. Vainshtein.
{\em Hurwitz numbers and intersections on moduli spaces of curves},
Invent. Math. {\bf 146} (2001), 297-327.

\bibitem{EMS}
B. Eynard, M. Mulase, B. Safnuk
{\em The Laplace transform of the cut-and-join equation and the Bouchard-Marino conjecture on Hurwitz numbers},
arXiv:0907.5224.


\bibitem{Eyn-OraInv}
B. Eynard, N. Orantin,
{\em Invariants of algebraic curves and topological expansion},
Commun. Number Theory Phs. {\bf 1}(2007), no.2, 347-452,
arXiv:math-ph/0702045.

\bibitem{Eyn-Ora}
B. Eynard, N. Orantin,
{\em Computation of open Gromov-Witten invariants for toric
Calabi-Yau 3-folds by topological recursion, a proof of the
BKMP conjecture},
arXiv:1205.1103.

\bibitem{Gou-Jac-Vai} I.P. Goulden, D. M. Jackson,A. Vainshtein,
{\em The number of ramified coverings of the sphere by the torus and surfaces
of higher genera}, Ann. Combinatorics {\bf 4} (2000), 27-46.


\bibitem{Guk-Sul}
S. Gukov, P. Su{\l}kowski,
{\em A-polynomial, B-model, and quantization},
JHEP1202(2012)070,
arXiv:1108.0002.

\bibitem{LLZ}  C.-C. Liu, K. Liu, J. Zhou,
{\em A proof of a conjecture of Mari\~no-Vafa on Hodge Integrals}, J. Differential Geom. {\bf 65} (2003),
289-340.

\bibitem{LLZ2} C.-C. Liu, K. Liu, J. Zhou, {\em A formula of two-partition Hodge integrals},
J. Amer. Math. Soc. {\bf 20} (2007), no. 1, 149-184.

\bibitem{LLLZ}
J. Li, C.-C. Liu, K. Liu, J. Zhou,
{\em A mathematical theory of the topological vertex},
arXiv:math/0408426.

\bibitem{Mar}
M. Mari\~no,
{\em Open string amplitudes and large order behavior in topological string theory},
arXiv:hep-th/0612127.

\bibitem{Mar-Vaf}
M. Mari\~no, C. Vafa,
{\em Framed knots at large N}, Orbifolds in mathematics and physics (Madison, WI,
2001), 185-204, Contemp. Math., 310, Amer. Math. Soc., Providence, RI, 2002.

\bibitem{Oko-Pan} A. Okounkov, R. Pandharipande,
{\em Hodge integrals and invariants of the unknot},
Geometry $\&$ Topology, Vol. {\bf  8} (2004), Paper no. 17, 675-699.

\bibitem{ZhoHH} J. Zhou, {\em Hodge integrals, Hurwitz numbers, and symmetric groups}, math.AG/0308024.

\bibitem{ZhoMV} J. Zhou,
{\em  Local mirror symmetry  for one-legged topological vertex},
arXiv:0910.4320.

\bibitem{ZhoCon}
J.~Zhou,
{\em Open string invariants and mirror curve of the resolved conifold},
arXiv:1001.0447.

\bibitem{ZhoFMV}
J.~Zhou
{\em A proof of the full Mari\~no-Vafa Conjecture},
Math. Res. Lett. {\bf 17} (2010), 1091-1099,
arXiv:1001.2092.

\bibitem{ZhoAiry}
J. Zhou,
{\em Intersection numbers on Deligne-Mumford moduli spaces
and quantum Airy curve},
arXiv:1206.5896.

\end{thebibliography}
\end{document}